\documentclass[a4paper]{amsart} 

\usepackage{geometry}
\usepackage[english]{babel}
\usepackage[utf8]{inputenc}

\usepackage{amssymb,amsmath,amsfonts,amsthm}
\usepackage{verbatim,url}
\usepackage{setspace}
\usepackage{hyperref}

\usepackage{tikz}
\usepackage{tikz-cd}
\usepackage[all,cmtip]{xy}
\usepackage{pgfplots}
\usepackage{verbatim}
\usepackage{longtable}
\usepackage{esvect}
\usetikzlibrary{math} 

\newtheorem{theorem}{Theorem}[section]

\newtheorem{Lemma}[theorem]{Lemma}

\newtheorem{remark}[theorem]{Remark}

\theoremstyle{remark}

\newcommand{\Aut}{\operatorname{Aut}}

\begin{document}

\title{Examples of  surfaces with canonical maps of degree $12$, $13$, $15$, $16$ and $18$}
\author{Federico Fallucca}

\thanks{
\textit{Acknowledgements:} The author would like to thank  Fabrizio Catanese, Davide Frapporti, Bin Nguyen and Roberto Pignatelli
 for useful comments and discussions.}
\address{Federico Fallucca 
\newline Università di Trento, Via Sommarive 14, I-38123 Trento (TN), Italia}
\email{federico.fallucca@unitn.it}

\begin{abstract}
In this note we present examples of complex algebraic surfaces with canonical maps of degree $12$, $13$, $15$, $16$  and $18$. They are constructed as quotients 
of a product of two curves of genus $10$ and $19$ using certain non-free actions of the group $S_3\times \mathbb Z_3^2$. To our knowledge there are no other examples in literature of surfaces with canonical map of degree $13$, $15$ and $18$.
\end{abstract}

\maketitle

\section{Introduction}

\noindent
Beauville has shown in \cite{B79} that if the image of the canonical map $\Phi_{K_S}$ of a surface has dimension $2$, then its degree $d$ is bounded as follows:  
\[
d:=\deg(\Phi_{K_S}) \leq 9+\frac{27-9q}{p_g-2}\leq 36.
\]
Note that the bound $d\leq 36$ was shown first by Persson in \cite[Proposition $5.7$]{Per}.
Here,  $q$ is the irregularity and $p_g$ the geometric genus of $S$. 
In particular, $28 \leq d$ is only possible if $q=0$ and $p_g=3$.
Motivated  by this observation, 
the construction  of surfaces with $p_g=3$ and canonical map of degree $d$ for every  value $2 \leq d \leq  36$ is an 
interesting, but still widely open  problem \cite[Question 5.2]{MLP21}. 
For a long time the only examples with $10\leq d$ were the surfaces of 
Persson \cite{Per},  with canonical map of degree $16$, and Tan \cite{Tan}, with degree $12$.  
In recent years, this
problem attracted the attention of  many authors,  putting an increased effort in the construction of new examples. 
 As a result,  we have now examples in literature for all degrees $2\leq d \leq  12$ and $d=14,16, 20, 24, 27, 32$ and  $36$, see \cite{MLP21},\cite{Ri15, Ri17, Ri17Zwei, Ri22}, \cite{LY21}, \cite{GPR}, \cite{Bin19, Bin21}, \cite{FG2022} and \cite{Bin22}.
  
 In this paper we construct surfaces as quotients of a product of two curves $C_1\times C_2$ modulo an action of the group $S_3\times \mathbb{Z}_3^2$. Here $C_1$ is a fixed curve of genus $10$ while $C_2$ is a curve of genus $19$ varying in a one-dimensional family. Varying the action of $S_3\times \mathbb{Z}_3^2$ we get four different one-dimensional families of canonical models of surfaces of general type with $K_S^2=24$, $p_g=3$ and $q=0$.

We write the canonical system of each of them in terms of  invariant holomorphic two-forms on the product $C_1\times C_2$. 
It turns out that for none of them  $\vert K_{S}\vert$ is base-point free, i.e. the canonical map 
$\Phi_{K_{S}} \colon S \dashrightarrow \mathbb P^2$ is just a rational map. To compute its degree, we resolve
 the indeterminacy by a sequence of blowups and compute  the degree of the resulting morphism  via elementary intersection theory.  It turns out that the degree of the canonical map is not always constant in a family and in fact it assumes five different values: $d=12,13,15,16$ and $18$. To our knowledge there are no other examples in literature of surfaces with canonical map of degree $13$, $15$ and $18$. \footnote{During the preparation of this work Bin Nguyen has communicated to us a different construction of a surface with canonical map of degree $13$.}

We point out that our surfaces are examples of product-quotient surfaces, i.e. quotients of product of two curves modulo an action of a finite group. In our cases the action is diagonal and non-free, arising surfaces with $8$ rational double points as singularities of type $\frac{1}{2}(1,1)$.
Product-quotient surfaces are studied for the first time by Catanese in \cite{Cat00}. They are revealed to being a very useful tool for building new examples of algebraic surfaces and studying  their geometry  in an accessible way. Apart from other works, that mainly deal with irregular surfaces, we want to mention the 
complete classification of surfaces isogenous to a product with $p_g=q=0$ \cite{BCG} and the classification  for $p_g=1$ and  $q=0$ under the assumption that the action is diagonal  \cite{G15}, 
the rigid but not infinitesimally rigid manifolds \cite{BP21} of Bauer and Pignatelli that gave a  negative answer to  a question of Kodaira and Morrow \cite[p.45]{KM71} and 
also the infinite series
 of $n$-dimensional infinitesimally rigid manifolds of general type with non-contractible universal cover for each $n\geq 3$,  provided  by 
 Frapporti and Gleissner\cite{FG}. 


\medskip
\noindent
{\bf Notation:} 
%
An algebraic surface $S$ is a \textit{canonical model} if it has at most rational double points as singularities and ample canonical divisor. Recall that each surfaces of general type is birational to a unique canonical model. In particular the minimal resolution of the singularities of $S$ is its minimal model. 

Let us denote by $\sigma$ and $\tau$ a rotation ($3$-cycle) and a reflection (transposition) of $S_3$ respectively. Consider also the three irreducible characters of $S_3$, so the trivial character $1$, the character $\textit{sgn}$ computing the sign of a permutation, and the only $2$-dimensional irreducible character $\mu:=\frac{1}{2}\left(\chi_{reg}-sgn-1\right)$, where $\chi_{reg}$ is the character of the regular representation of $S_3$.
\\
Let us fix a basis $e_1, e_2$ of $\mathbb Z_3^2$  and consider the dual characters $\epsilon_1$, $\epsilon_2$ of $e_1$ and $e_2$, i.e. the characters defined by
\[
\epsilon_i(ae_1+be_2):=\zeta_3^{a\delta_{1i}+b\delta_{2i}}, \qquad \zeta_3:=e^{\frac{2\pi i }{3}},
\]
where $\delta_{ij}$ is the Kronecker delta. 
\\
Given a representation $\rho$ on a vector space $V$ and an isotypic component $W$ of $V$ of character $\chi$, we can sometimes write $W_\chi$ instead of  $W$ for specifying its character;
\\
When we write $\sqrt[n]{\lambda}$ we mean one of the $n$-roots (arbitrarily chosen) of the complex number $\lambda$. 
\\
Finally, denote by $[j]\in \{0,1\}$ the class of the integer number $j$ modulo $2$.

\section{The surfaces}

\noindent
In this section we construct a series of surfaces $S$,  
 as quotients of a product of the two curves $C_1$ and $C_2$, 
modulo a suitable diagonal  action of the group $S_3\times \mathbb Z_3^2$. For any surface $S$, 
we determine the canonical map $\Phi_{K_S}$ and compute its  degree. 

We consider the projective space $\mathbb{P}^3$ with homogeneous coordinates $x_0, \ldots, x_3$ and the weighted projective space $\mathbb{P}^3(1,1,1,2)$ with homogeneous coordinates $y_0, \ldots, y_3$. Here $y_3$ is the variable of weight $2$.
We take the curves $C_1\subseteq \mathbb{P}^3$ and $C_2\subseteq \mathbb{P}^3(1,1,1,2)$ as follows 
\[
C_1 \colon \begin{cases}
	x_2^3=x_0^3-x_1^3 \\
	x_3^3=x_0^3+x_1^3
\end{cases}, \qquad   C_2 \colon \begin{cases}
	y_2^3=y_0^3+y_1^3 \\
	y_3^3=y_0^6+y_1^6-2\lambda y_0^3y_1^3
\end{cases}, \lambda\neq -1,1
\]
 
\noindent 
Both curves are smooth,  in fact this is the reason why we assume $\lambda \neq -1,1$ in the definition of $C_2$.

On the first curve $C_1$ we consider the action of $S_3\times \mathbb Z_3^2$ given by 
\[
\phi_1\colon  S_3\times \mathbb Z_3^2 \to \Aut(C_1), \quad \left(\sigma^i\tau^j, (a,b)\right) \mapsto [(x_0:x_1:x_2: x_3) \mapsto (\zeta_3^ix_{[j]}:x_{[j+1]} : (-1)^j\zeta_3^{2a+2i} x_2:  \zeta_3^{2b+2i} x_3)]. 
\]
We leave to the reader to checking that this defines an action. 

Note that the automorphisms $\phi_1(\sigma^i\tau^j,(a,b))$ are precisely the   deck transformations of the  cover
\[
\pi_1 \colon C_1 \stackrel{9 : 1}{\longrightarrow} \mathbb P^1 \stackrel{6 : 1}{\longrightarrow} \mathbb P^1, \qquad (x_0:x_1:x_2:x_3) \mapsto (x_0:x_1)\mapsto \left(x_0^3x_1^3: (x_0^6+x_1^6)/2\right).
\]
In particular 
$C_1/\left(S_3\times \mathbb Z_3^2\right) \simeq \mathbb P^1$ and $\pi_1$ is the quotient map. 
The cover   is branched along $p_1:=(1:1)$, $p_2:=(0:1)$ and $p_3:=(-1:1)$, corresponding to the three orbits of the points  with non trivial stabilizer, of respective length $9,18$ and $9$. A representative of each orbit and a generator of the stabilizer is given by: 
\[
\begin{tabular}{c | c | c | c }
 & $p_1$ & $p_2$ & $p_3$   \\
 \hline
\makebox{representative} & $(1:1:0:\sqrt[3]{2})$ & $(1:0:1:1)$  & $ (1: -\zeta_3 :\sqrt[3]{2}:0)$   \\
\hline
\makebox{generator} & $ g_1:=\left(\tau,(1,0)\right) $ & $g_2:=(\sigma^2,(2,2)) $  & $g_3:=(\sigma\tau,(0,1)) $  \\
\end{tabular}
\]

\noindent 
On the second curve $C_2$ the action $\phi_2$ is defined as

\[
\phi_2 \colon S_3\times \mathbb Z_3^2 \to \Aut(C_2), \quad \left(\sigma^i\tau^j, (a,b)\right) \mapsto [(y_0:y_1:y_2: y_3) \mapsto (\zeta_3^iy_{[j]}:y_{[j+1]} : \zeta_3^{a+2b+2i} y_2:  \zeta_3^{2a+2b+i} y_3)]. 
\]
As in the previous case, we leave to the reader to checking that this defines a group action and note that the automorphisms $\phi_2(\sigma^i\tau^j,(a,b))$ are precisely the deck transformations of the  cover
\[
\pi_2\colon C_2 \stackrel{9:1}{\longrightarrow}\mathbb P^1 \stackrel{6:1}{\longrightarrow} \mathbb P^1, \qquad (y_0:y_1:y_2:y_3) \mapsto (y_0:y_1)\mapsto \left(y_0^3y_1^3: (y_0^6+y_1^6)/2\right).
\]
Hence
$C_2/\left(S_3\times \mathbb Z_3^2\right) \simeq \mathbb P^1$ and $\pi_2$ is the quotient map. The cover is branched along $q_1:=(1:1)$, $q_2:=(0:1)$, $q_3:=(1:\lambda)$ and $q_4:=(-1:1)$, corresponding to the four orbits of the points  with non trivial stabilizer, of respective length $27,18, 18$ and $9$. Note that the points $q_j$ are pairwise distinct under the assumption $\lambda \neq -1,1$.


A representative of each orbit and a generator of the stabilizer is given by: 
\[
\small\begin{tabular}{c | c | c | c |c}
	 & $q_1$ & $q_2$  & $ q_3$ & $q_4$   \\
	\hline
	\makebox{representative} & $(1:\zeta_3:\sqrt[3]{2}:\sqrt[3]{2-2\lambda})$ & $(0:1:1:1)$  & $ (1: \sqrt[3]{\lambda-\sqrt{\lambda^2-1}}:\sqrt[3]{1+\lambda-\sqrt{\lambda^2-1}}:0)$ & $(1:-1:0:\sqrt[3]{2+2\lambda})$   \\
	\hline
	\makebox{generator} & $ h_1:=\left(\sigma\tau,0\right) $ & $h_2:=(\sigma,(1,0)) $  & $h_3:=(Id,(1,1)) $ & $h_4:=(\tau,(1,2))$ \\
\end{tabular}
\]
We compute the action of $S_3\times \mathbb{Z}_3^2$ on $H^0(C_i,\Omega_{C_i}^1)$.

By standard adjunction theory $H^0(C_1,\Omega_{C_1}^1)$ is isomorphic to $H^0(C_1,\mathcal{O}_{C_1}(2))$, isomorphism mapping a monomial $x_0^{2-\alpha-\beta-\gamma}x_1^\alpha x_2^\beta x_3^\gamma$ to the $1$-form $\omega_{\alpha\beta\gamma}$ that in affine coordinates is 
\[
\omega_{\alpha\beta\gamma}:=u^\alpha v^{\beta-2} t^{\gamma-2}du, \qquad \makebox{where} \qquad u:=\frac{x_1}{x_0} \quad   v:=\frac{x_2}{x_0} \quad \makebox{and}  \qquad t:=\frac{x_3}{x_0}.
\]
The character of the \textit{canonical} representation of $C_1$, the action of  $S_3\times \mathbb Z_3^2$ on $H^0(C_1,\Omega_{C_1}^1)$, can be computed by the standard Chevalley-Weil formula and is amount to
\[
\chi_{can}^1=\epsilon_1^2\cdot\epsilon_2^2+sgn\cdot\epsilon_1\cdot\epsilon_2+sgn\cdot\epsilon_2+sgn\cdot\epsilon_1+\mu\cdot\epsilon_1\cdot\epsilon_2+\mu\cdot\epsilon_1^2\cdot\epsilon_2+\mu\cdot \epsilon_1\cdot\epsilon_2^2.\]
We give an explicit decomposition in irreducible subspaces. Using the expression in affine coordinates we obtain 
\[
\begin{split}
	(\sigma^i\tau^j,(a,b))\cdot \omega_{\alpha\beta\gamma} & = \phi_1(\left(\sigma^i\tau^j,(a,b)\right)^{-1})^{\ast}(\omega_{\alpha\beta\gamma}) \\ & =(-1)^{j(\beta-1)}\zeta_3^{a(\beta-2)+b(\gamma-2)+\left(\alpha-(2\alpha+\beta+\gamma-2)[j]+2\beta+2\gamma-7\right)i}\omega_{(\alpha-(2\alpha+\beta+\gamma-2)[j])\beta\gamma}. 
\end{split}
\]
A tedious but straightforward computation gives the following decomposition: 
\[
\begin{split}
	H^0(C_1,\Omega_{C_1}^1)= & \langle \omega_{011}\rangle_{\epsilon_1^2\cdot\epsilon_2^2} \oplus \langle \omega_{100}\rangle_{sgn\cdot \epsilon_1\cdot \epsilon_2} \oplus \langle \omega_{020}\rangle_{sgn\cdot \epsilon_2} \oplus \langle \omega_{002}\rangle_{sgn\cdot \epsilon_1} \oplus \\
	& \langle \omega_{000},\omega_{200}\rangle_{\mu\cdot \epsilon_1\cdot \epsilon_2}\oplus \langle \omega_{010},\omega_{110}\rangle_{\mu\cdot\epsilon_1^2\cdot\epsilon_2}\oplus \langle \omega_{001},\omega_{101}\rangle_{\mu\cdot \epsilon_1\cdot \epsilon_2^2}.
\end{split}
\]
Similarly, adjunction theory gives an isomorphism among $H^0(C_2,\Omega_{C_2}^1)$ and $H^0(C_2,\mathcal{O}_{C_2}(4))$ mapping a monomial $y_0^{4-\alpha-\beta-2\gamma}y_1^\alpha y_2^\beta y_3^\gamma$ to the $1$-form $\omega'_{\alpha\beta\gamma}$ that in affine coordinates is 
\[
\omega'_{\alpha\beta\gamma}:=(u')^{\alpha}(v')^{\beta-2} (t')^{\gamma-2}du', \qquad \makebox{where} \qquad u':=\frac{y_1}{y_0} \quad   v':=\frac{y_2}{y_0} \quad \makebox{and}  \qquad t':=\frac{y_3}{y^2_0}.
\]
We obtain a basis of $19$ dimension space $H^0(C_2, \mathcal{O}_{C_2}(4))$ by taking the $22$ monomials of degree $4$ in the variables $y_j$ and removing $y_0y_2^3$, $y_1y_2^3$ and $y_2^4$, that can be expressed in terms of the other monomials using the cubic equation defining $C_2$. 
Accordingly we get a basis of $H^0(C_2,\Omega_{C_2}^1)$ by removing from that set $\omega'_{\alpha\beta\gamma}$ the $1$-forms
$\omega'_{040}, \omega'_{030}$ and $\omega'_{130}$.
The \textit{canonical} character of $C_2$ is given by Chevalley-Weil as 
\[
\chi_{can}^2=sgn \cdot \epsilon_1^2\cdot \epsilon_2+sgn\cdot\epsilon_1^2\cdot \epsilon_2^2+sgn\cdot \epsilon_1\cdot \epsilon_2+sgn\cdot \epsilon_1+sgn\cdot \epsilon_2^2+\mu\cdot \epsilon_1+\mu\cdot\epsilon_2+2\mu\cdot \epsilon_2^2+sgn\cdot\epsilon_1^2+\epsilon_1^2+\mu\cdot \epsilon_1^2+\mu\cdot\epsilon_1\cdot\epsilon_2,
\]
and the action on $H^0(C_2,\Omega_{C_2}^1)$ computed in affine coordinates as above is 
\[
\begin{split}
	(\sigma^i\tau^j,(a,b))\cdot \omega'_{\alpha\beta\gamma} & =  \phi_2(\left(\sigma^i\tau^j,(a,b)\right)^{-1})^{\ast}(\omega'_{\alpha\beta\gamma}) \\ &
	=(-1)^{j}\zeta_3^{a(2\beta+\gamma)+b(\beta+\gamma-4)+\left(\alpha-(2\alpha+\beta+2\gamma-4)[j]+2\beta+\gamma+1\right)i}\omega'_{(\alpha-(2\alpha+\beta+2\gamma-4)[j])\beta\gamma}.
\end{split}
\]
Another tedious computation gives the decomposition 
\[
  \begin{split}
  	H^0(C_2,\Omega_{C_2}^1)= & \langle \omega'_{002}\rangle_{sgn\cdot \epsilon_1^2\cdot \epsilon_2} \oplus \langle \omega'_{021}\rangle_{sgn\cdot \epsilon_1^2\cdot \epsilon_2^2}\oplus \langle \omega'_{120}\rangle_{sgn\cdot\epsilon_1\cdot\epsilon_2}  \\
  	& \oplus\langle \omega'_{101}\rangle_{sgn\cdot\epsilon_1}\oplus \langle \omega'_{200}\rangle_{sgn\cdot \epsilon_2^2}\oplus \langle \omega'_{001},\omega'_{201}\rangle_{\mu\cdot \epsilon_1} \oplus \langle\omega'_{011},\omega'_{111}\rangle_{\mu\cdot \epsilon_2} \\ &
  	\oplus \left(\langle\omega'_{000},\omega'_{400}\rangle\oplus \langle \omega'_{100},\omega'_{300}\rangle\right)_{\mu\cdot \epsilon_2^2}\oplus \langle\omega'_{010}+\omega'_{310}\rangle_{sgn\cdot \epsilon_1^2}\oplus\langle\omega'_{010}-\omega'_{310}\rangle_{\epsilon_1^2}
  	\\&
  	\oplus \langle \omega'_{110},\omega'_{210}\rangle_{\mu\cdot \epsilon_1^2} \oplus \langle \omega'_{220},\omega'_{020}\rangle_{\mu\cdot \epsilon_1\cdot \epsilon_2}.
  \end{split}
\]
\bigskip

We consider unmixed quotients $S:=(C_1\times C_2)/\left( S_3\times \mathbb Z_3^2\right)$ modulo a diagonal action $\phi_1\times \left(\phi_2\circ \Psi\right)$, where $\Psi$ is one of the automorphism of $S_3\times \mathbb{Z}_3^2$. 
\\
Firstly we study the singularities of $S$. We observe that $C_1$ and $C_2$ have stabilizers of order $6,3 $ and $6$ and $2,3,3$ and $6$ respectively. Hence $18$ points of $C_1$ and $36$ points of $C_2$ have stabilizer of even order. However $S_3\times \mathbb{Z}_3^2$ has only three elements of order $2$ and they are in the same conjugacy class. This means that each of these three elements fix exactly $6\cdot 12=72$ points of $C_1\times C_2$. Thus $S$ can never be smooth and if it admits only nodes, then they are in total $3\cdot 72 /27=8$. 
\\
Now let us consider the following automorphisms of $S_3\times \mathbb{Z}_3^2$
\begin{equation}\label{automorfismi}
\begin{aligned}
	\Psi_1 & =Id, & \Psi_2 &= \left(\begin{cases}
		\sigma \mapsto \sigma \\
		\tau \mapsto \tau\sigma \\
	\end{cases}, \begin{pmatrix} 0&1 \\ 2 & 0\end{pmatrix} \right), &\\
\Psi_3 &=  \left(\begin{cases}
		\sigma \mapsto \sigma^2 \\
		\tau \mapsto \tau \\
	\end{cases}, \begin{pmatrix} 0&2 \\1 & 0\end{pmatrix} \right),& \Psi_4 &= \left(\begin{cases}
		\sigma \mapsto \sigma^2 \\
		\tau \mapsto \tau \\
	\end{cases}, \begin{pmatrix} 0&2 \\ 2 & 0\end{pmatrix} \right).&
\end{aligned}
\end{equation}
A direct computation shows us that for these four choices of $\Psi$ the surface $S$ has exactly $8$ nodes and no other singularities.
\begin{remark}
	The first example has been  found by using the database \cite{CGP22}. 
	Later on we have run a systematic research over all automorphisms of $S_3\times \mathbb{Z}_3^2$ proving that the obtained surfaces having only nodes are isomorphic to the four  surfaces presented in this note.
\end{remark}

The vector space $H^0(K_{S})$ is isomorphic to the invariant subspace 
$\big(H^0(\Omega_{C_1}^1) \otimes H^0(\Omega_{C_2}^1) \big)^{S_3\times \mathbb Z_3^2}$, 
where  the action on the tensor product is diagonal, i.e.  $\left(\sigma^i\tau^j,(a,b)\right)\in S_3\times \mathbb Z_3^2$ acts via 
\begin{equation}\label{azione_twistata}
	\phi_1(\left(\sigma^i\tau^j,(a,b)\right)^{-1})^{\ast} \otimes  \phi_2(\Psi(\left(\sigma^i\tau^j,(a,b)\right)^{-1}))^{\ast}. 
\end{equation}

\noindent 
For each character $\eta$ of $S_3\times \mathbb{Z}_3^2$ define its twist by $\Psi$ as 
\[
\eta_\Psi:=\eta\circ \Psi^{-1}.
\]
Pulling back $H^0(K_S)$ to $C_1\times C_2$  we obtain
\begin{Lemma}\label{invariantforms}
A basis of $H^0(K_S)$ is given by the $\left(S_3\times \mathbb Z_3^2\right)$-invariant $2$-forms of  $H^0(\Omega_{C_1}^1) \otimes H^0(\Omega_{C_2}^1) $ with respect to the action \eqref{azione_twistata}. Hence
\[
\big(H^0(\Omega_{C_1}^1) \otimes H^0(\Omega_{C_2}^1) \big)^{S_3\times \mathbb Z_3^2}=\bigoplus_{\eta\neq 0} \big(H^0(\Omega_{C_1}^1)_{\eta}\otimes H^0(\Omega_{C_2}^1)_{\overline{\eta_\Psi}}\big)^{S_3\times \mathbb Z_3^2},
\]
where $H^0(\Omega_{C_i}^1)_{\eta}$ is the isotypic component of $H^0(\Omega_{C_i}^1)$ of character $\eta$. Moreover 
\[p_g=\langle \chi_{can}^1\cdot \chi_{can}^2,1\rangle=\sum_{\eta\neq 0} \langle \chi_{can}^1,\eta \rangle \cdot \langle \chi_{can}^2, \overline{\eta_\Psi}\rangle.
\]
\end{Lemma}
\noindent 
Denote by $\omega_{jklmrs}:=\omega_{jkl}\otimes \omega'_{mrs}$. We can now state and prove our main result: 
\begin{theorem}\label{MainTheo}
For all  $\Psi \in \Aut(S_3\times \mathbb Z_3^2)$ in  \eqref{automorfismi}, 
the diagonal action $\phi_1 \times (\phi_2 \circ \Psi)$ of $S_3\times \mathbb Z_3^2$ on the product of the two curves $C_1$ and $C_2$ is not free. The quotient is a canonical model of a regular surface $S$ of general type with $K_S^2=24$, $p_g=3$ and with $8$ rational double points as singularities of type $\frac{1}{2}(1,1)$.  
A basis of $H^0(K_S)$, the canonical map $\Phi_{K_S}$ in projective coordinates and its degree are stated  in 
the table:

\[
\Small\begin{tabular}{c | c | c | c | c }
\makebox{No} & $\Psi$ & \makebox{Basis of $H^0(K_S)$}  & $\Phi_{K_S}(x,y)$ & $\deg(\Phi_{K_S})$  \\
\hline
1. & $ Id $ & $\lbrace \omega_{100021},\omega_{020200}, \omega_{002040}  \rbrace$ &   
$(x_0x_1y_2^2y_3: x_2^2y_0^2y_1^2: x_3^2y_2^4)$ & $18$ \\ 
\hline
2. & $ \Psi_2$ & $\lbrace    
\omega_{020101},
\omega_{002200},
\zeta_3\omega_{010020}-\omega_{110220} \rbrace$ &  
$(x_2^2y_0y_1y_3:x_3^2y_0^2y_1^2: x_2y_2^2(\zeta_3x_0y_0^2-x_1y_1^2))$ & $\begin{cases}
	15 \quad \makebox{if} \quad \lambda \neq 0 \\
	13 \quad \makebox{if} \quad \lambda = 0 \\
 \end{cases}$ \\
\hline
3. & $ \Psi_3$ & 
$\lbrace   
\omega_{100002},
\omega_{020040}, 
\omega_{001220}+\omega_{101020} \rbrace$  &   
$(x_0x_1y_3^2:x_2^2y_2^4:x_3y_2^2(x_0y_1^2+x_1y_0^2))$  & $\begin{cases}
	18 \quad \makebox{if} \quad \lambda \neq 0 \\
	16 \quad \makebox{if} \quad \lambda = 0 \\
\end{cases}$  \\ 
\hline
4. & $\Psi_4$ & 
$\lbrace   
\omega_{100120},
\omega_{020101}, 
\omega_{000020}+\omega_{200220} \rbrace$  &   
$(x_0x_1y_0y_1y_2^2:x_2^2y_0y_1y_3:y_2^2(x_0^2y_0^2+x_1^2y_1^2))$  & $12$  \\ 
\hline
\end{tabular}
\]
\end{theorem}

\begin{proof}
We have already mentioned that for all $\Psi$ in \eqref{automorfismi} the action is not free and the quotient $S$ has $8$ singularities of type $\frac{1}{2}(1,1)$ and no other singularities. The genus of the two curves is $g(C_i)\geq 2$, hence $C_1\times C_2$ has ample canonical divisor and so $S$ has ample canonical divisor too. It follows $S$ is a canonical model. 

The self-intersection of  the canonical divisor  of each $S$ is amount to
\[
K_S^2=\frac{8(g(C_1)-1)(g(C_2)-1)}{\vert S_3\times \mathbb Z_3^2\vert }=24.
\]

They are regular surfaces, because they do not possess any non-zero holomorphic one-forms, since  $C_i/\left(S_3\times \mathbb Z_3^2\right)$ is biholomorphic to $\mathbb P^1$. 
The geometric genus of each  $S$ is therefore equal to  (compare \cite{BP12})
\[
p_g=  \chi(\mathcal O_{S})- 1 = \frac{(g(C_1)-1)(g(C_2)-1)}{\vert S_3\times \mathbb  Z_3^2\vert}+\frac{1}{12}\left(8\cdot \frac{3}{2}\right)-1=3. 
\]
Using Lemma \ref{invariantforms} we have computed a basis of $H^0(K_S)$. In fact since we have proved that $p_g=3$ it is enough to verify that the given elements of the table are invariant for the corresponding action. Applying the explicit isomorphisms from $H^0(C_1,\Omega_{C_1}^1)$ to $H^0(C_1,\mathcal{O}_{C_1}(2))$ and from   $H^0(C_2,\Omega_{C_2}^1)$ to $H^0(C_2,\mathcal{O}_{C_2}(4))$ we obtain the 
product of  quadrics and quartics defining  the canonical map in the table.

It remains to determine the degree of $\Phi_{K_S}$ for each surface $S$. Instead to work on $S$ it is convenient to work on $C_1\times C_2$, that is smooth: 
\[
\xymatrix{
	C_1\times C_2 \ar[r]^{\lambda_{12}}\ar[dr]_{\Phi_{K_{C_1\times C_2}}}& S  \ar@{-->}[r]^{\Phi_{K_S}} & \mathbb{P}^2 \ar@{<--}[dl]\\
	& \mathbb P^{10\cdot 19 -1}.
}
\]
Note that the map $\Phi_{K_S}\circ \lambda_{12}$ is induced by the sublinear system $\vert T \vert $ 
  of $\vert K_{C_1\times C_2}\vert$ generated by the three invariant $2$-forms defining $\Phi_{K_S}$.  In particular the self-intersection of  $T$ is amount to 
  \[
  T^2=\left(\lambda_{12}^* K_S\right)^2=\vert S_3\times \mathbb  Z_3^2\vert \cdot K_S^2=54\cdot 24.
  \]
We 
\emph{resolve  the indeterminacy} of  $\Phi_T=\Phi_{K_S}\circ \lambda_{12}$ 
by a sequence of  blowups, as explained in the textbook \cite[Theorem II.7]{Beauville}:

\[
\xymatrix{
\widehat{C_1\times C_2} \ar[r] \ar[dr]_{\Phi_{\widehat{M}}} & C_1\times C_2\ar@{-->}[d]^{\Phi_{T}} \\ 
& \mathbb P^2.
}
\]
Here the morphism $\Phi_{\widehat{M}}$ is induced by the base-point free linear system $\vert \widehat{M} \vert$ obtained as follow:
\\
We blow up the base-points of  $\vert T\vert$, take the pullback of the mobile part $\vert M \vert$  of $\vert T\vert$
and remove the fixed part of this new linear system. We repeat the procedure, until we obtain a  base-point free  linear system $\vert\widehat{M}\vert $.  

The     self-intersection $\widehat{M}^2$ is positive if and only if 
$\Phi_{\widehat{M}}$ is not composed by a pencil. In this case $\Phi_{\widehat{M}}$  is onto and  
 it holds: 
\[
\deg(\Phi_{K_S})=\frac{1}{\vert S_3\times \mathbb{Z}_3^2 \vert }\deg(\Phi_{\widehat{M}})=\frac{1}{54}\widehat{M}^2. 
\]
For the computation of the resolution,  it is convenient to write the divisors of the  product of quadrics and quartics defining  $\Phi_{K_S}$ (and hence $\Phi_T$
) as linear combinations of the 
curves  $F_j:=\lbrace x_j=0\rbrace$ and $G_k:=\lbrace y_k=0\rbrace$ on $C_1\times C_2$. We point out that these curves are reduced and intersect pairwise transversally thanks to the assumption $\lambda \neq -1,1$.
In particular $(F_j,F_k)=(G_j,G_k)=0$ and $(F_j, G_k)=81$,  for  $k\neq 3$, while $(F_j, G_3)=162$. \\
Consider the first surface in the table. 
Here, the divisors of the three products of quadrics and quartics spanning  the subsystem  $\vert T\vert$ are: 
\[
F_0+F_1+ 2G_2+G_3, \qquad 
2F_2+2G_0+2G_1 \qquad \makebox{and} \qquad 
2F_3+4G_2. 
\]
Here $\vert T\vert$ has not fixed part and it has precisely $81$ (non reduced) base-points $F_2\cap G_2$.
We can perform the computation of  the difference $T^2- \widehat{M}^2$ by applying Lemma \ref{FedericoLemma} below (for a proof see \cite[Lemma  2.3] {FG2022}) recursively for each base-point of $\vert T \vert$:

\begin{Lemma}\label{FedericoLemma} 
	Let  $\vert M \vert $ be a two-dimensional linear system on a surface $S$ spanned by $D_1$, $D_2$ and $D_3$. Assume that $\vert M \vert $ has only isolated base-points, smooth for $S$, and that in a neighborhood of a basepoint $p$ we can write  the divisors 
	$D_i$ as 
	\[
	D_1=aH, \quad D_2=bK \quad \makebox{and} \quad D_3= cH+d K.
	\]
	Here $H$ and $K$ are reduced, smooth and intersect transversally at $p$ and $a,b,c,d$ are non-negative integers, $b\leq a$.
	Assume that 
	\begin{itemize}
		\item $d\geq b$ or
		\item $b\neq 0$ and $c+md\geq a$, where $a=mb+q$ with $0\leq q<b$.
	\end{itemize}
	Then after blowing up at most $(ab)$-times we obtain a new linear system $\vert \widehat{M} \vert $ such that no infinitely near point of 
	$p$ is a base-point of $\vert \widehat{M} \vert $. Moreover $\widehat{M}^2 =M^2-ab$.

\end{Lemma}

 In a neighbourhood of each of these base-points the three divisors are respectively 
\[
2G_2, \qquad 2F_2 \qquad \makebox{and} \qquad 4G_2. 
\]
Since $F_2$ and $G_2$ are transversal we are in the situation of the Lemma \ref{FedericoLemma} with $H=G_2$ and $K=F_2$, $a=b=2$ and $c=4$, $d=0$.
 So  $b\neq 0$ and $c+md\geq a$ and the Lemma applies.   The correction term is $ab=4$ for each of the $81$ base-points. Thus 
 \[
    T^2- \widehat{M}^2=4\cdot 81.
 \]
  The degree of the canonical map is therefore given by 
   \[
  \deg(\Phi_{K_S})=\frac{1}{54}\widehat{M}^2=\frac{1}{54}\left(T^2 - (T^2-\widehat{M}^2)\right)=\frac{1}{54}\left(54 \cdot  24- 4\cdot 81\right)=18. 
  \]
  
 Now we take in exam the second surface  in our table. Here the subsystem  $\vert T\vert $ is spanned  by: 
 \[
 D_1:= 2F_2+G_0+G_1+G_3, \quad 
 D_2 :=2F_3+2G_0+2G_1 \quad \makebox{and} \quad 
 D_3:= F_2+2G_2+\Delta,
 \]
 where $\Delta=(\zeta_3x_0y_0^2-x_1y_1^2)$. The (set-theoretical) base locus is 
 \[
 F_2\cap G_0, F_2\cap G_1, \quad \Delta \cap G_0, \Delta\cap G_1,  \quad \makebox{and}   \quad \Delta \cap F_3\cap G_3.
 \]
 We remark that the other pieces of the base locus are empty.  In fact that points would belong in some $F_i\cap F_j$ or $G_i\cap G_j$ and we have already mentioned that they are pairwise disjoint. 
 
  We determine the correction term to the self intersection number for each of these  base-points of $\vert T\vert$. 
  
  We consider first the $81$ points $F_2\cap G_i$, for $i=0,1$. Here $F_2$ and $G_i$ intersect transversally on each of them. Around one of these points, the  divisors $D_k$ are given by $G_i+2F_2$, $2G_i$ and $F_2$. We are in the situation of the Lemma with $H=G_i$ and $K=F_2$, $a=d=2$ and $b=c=1$. Hence $d\geq b$ and the Lemma applies, which yields $ab=2$ as correction term. 
  
 We let go on to the $81$ base-points $\Delta\cap G_i$. The local coordinates around one of these points are $X:=x_j/x_i$ and $Y:=y_i/y_j$, where $j=0,1, j\neq i$. So the divisors $D_k$ are respectively given by 
  \[
  \lbrace Y=0\rbrace , \qquad 2\lbrace Y=0 \rbrace \qquad \makebox{and} \qquad \lbrace \zeta_3^{1+i}Y^2-X=0 \rbrace.
  \]
  Thus $D_1$ and $D_3$ intersect transversally in $(0,0)$ and we fall down once more in the situation of the Lemma. Here $H=D_3$ and $K=D_1$, $a=b=1$, $c=0$ and $d=2$. Since $d\geq b$ then the Lemma is fulfilled  and the correction term is amount to $ab=1$.
  
  We consider finally the points $\Delta\cap F_3\cap G_3$. These points satisfy the equations  
\begin{equation}\label{equazioni}
	\begin{cases}
		y_3^3 =y_0^6+y_1^6-2\lambda y_0^3y_1^3 &=0 \\
		x_3^2 =x_0^3+x_1^3  & =0 \\
		\zeta_3x_0y_0^2-x_1y_1^2  & =0
	\end{cases}.
\end{equation}
  The last two equations imply that $x_1^3=-x_0^3$ and
  \[
  	 x_0^3y_0^6 =(\zeta_3x_0y_0^2)^3=(x_1y_1^2)^3 =x_1^3y_1^6=-x_0^3y_1^6.
  \] 
  Thus $y_0^6+y_1^6=0$ and comparing it with the first equation of  \ref{equazioni} we get $\lambda y_0^3y_1^3=0$.  Therefore  $\Delta\cap F_3\cap G_3$ is non empty only if $\lambda=0$. 
  \\
  Let us suppose $\lambda \neq 0$. Then 
  \[
  T^2-\widehat{M}^2=2\cdot 2 \cdot 81+2\cdot 81=6\cdot 81, 
  \]
  and the degree of the canonical map is amount to
 \[
 \deg(\Phi_{K_S})=\frac{1}{54}\left(T^2 - (T^2-\widehat{M}^2)\right)=\frac{1}{54}\left(54\cdot   24- 6\cdot 81\right)=15. 
  \]
 It remains to consider the case when $\lambda=0$.  The base-points $\Delta\cap F_3\cap G_3$ are the following $54$ ones:
  \[
  t_k:= \left(\left(1:-\zeta_3^{k_1}:\sqrt[3]{2}\zeta_3^{k_2}:0\right),\left(1:e^{\frac{\pi i  }{6}}\zeta_6^{k_3}:\sqrt[6]{2}e^{\frac{\pi i}{12}\left(1-2[k_3]\right)}\zeta_3^{k_4}:0\right)\right), \qquad k_1+k_3\equiv 2 \mod 3,
  \]
  where $k_i=0,1,2$, for $i\neq 3$, and $k_3=0, \dots, 5$.
  Fix coordinates $X:=x_1/x_0+\zeta_3^2$ and $Y:=y_1/y_0-e^{\frac{\pi i }{6}}$ around one of  these points, for example that one for $k=(2,0,0,0)$. The divisors $D_k$ are locally given by
  \[
  \lbrace Y=0 \rbrace, \qquad 2\lbrace X=0 \rbrace \qquad \makebox{and} \qquad \lbrace Y(2e^{\frac{\pi i 5}{6}}+Y-2e^{\frac{\pi i 5}{6}}X-XY)=0\rbrace=\lbrace Y=0\rbrace.
  \]
  In this case $H=\{X=0\}$ and $K=\{Y=0\}$ and  $a=2$ and $b=d=1$, $c=0$. The correction term is $ab=2$. \\
  Hence 
  \[
  T^2-\widehat{M}^2=2\cdot 2 \cdot 81+2\cdot 81+2 \cdot 54=6\cdot 81+2\cdot 54.
  \]
  The degree of the canonical map is therefore given by 
  \[
  \deg(\Phi_{K_S})=\frac{1}{54}\left(T^2 - (T^2-\widehat{M}^2)\right)=\frac{1}{54}\left(54\cdot   24- 6\cdot 81-2\cdot 54\right)=13.
  \]
  We leave to the reader to verifying with the same approach that the degree of the canonical map of the remain two surfaces are amount to that ones stated in the table. 
%
\end{proof}


\noindent 


\bigskip
\bigskip


\begin{thebibliography}{99}
\bibitem[BCG08]{BCG}
  I. Bauer, F.  Catanese, F. Grunewald,
\textit{The classification of surfaces with $p_g=q=0$ isogenous to a product of curves},
   Pure Appl. Math. Q.,
   \textbf{4}, 547--586, (2008).
%
\bibitem[BC18]{BC}
I. Bauer and F. Catanese, \textit{On rigid compact complex surfaces and manifolds}, Adv. Math. \textbf{333},  620–-669,  (2018).
%
\bibitem[BGV15]{BGV15}
I. Bauer, S. Garion, A. Vdovina (\textit{Editors}),
\textit{Beauville surfaces and groups}, Springer Proceedings in Mathematics \& Statistics, Proceedings of the conference held at the {U}niversity of
              {N}ewcastle, {N}ewcastle-upon-{T}yne, {J}une 7--9, 2012,
  \textbf{123}, (2015).            
 %
 \bibitem[BP12]{BP12}
 I. Bauer, R. Pignatelli,
 \textit{The classification of minimal product-quotient surfaces with
 	{$p_g=0$}}, Math. Comp.,
 \textbf{81},
 2389--2418
 (2012).
 %
\bibitem[BP21]{BP21}
   I. Bauer, R. Pignatelli,
   \textit{Rigid but not infinitesimally rigid compact complex manifolds}, Duke Math. J.,
   \textbf{170},
   1757--1780,
   (2021).
%
\bibitem[B79]{B79}
 A. Beauville, \textit{ L'application canonique pour les surfaces de type \'{g}eneral}. Invent. Math., 
 \textbf{55}, 121--140,
(1979).
%
\bibitem[B96]{Beauville} A. Beauville, Arnaud, 
\textit{Complex algebraic surfaces},
 London Mathematical Society Student Texts,
\textbf{34}, (1996). 
%
\bibitem[BCP97]{MAGMA} 
W. Bosma, J. Cannon, 
\textit{The Magma algebra system. I. The user language},
Computational algebra and number theory (London, 1993), Journal of Symbolic Computation,
\textbf{24}, 235--265, (1997). 
 %
 \bibitem[Cat00]{Cat00}
 F. Catanese,
 \textit{Fibred surfaces, varieties isogenous to a product and related
 	moduli spaces},
 American Journal of Mathematics, 
 \textbf{122}, 1--44, (2000)
%
\bibitem[CGP22]{CGP22} D. Conti, A.Ghigi, R. Pignatelli,
\textit{Topological types of actions on curves}, 	arXiv:2207.00981 , (2022)
%
 \bibitem[FG22]{FG2022} F. Fallucca and C. Gleissner, 
 \textit{Some surfaces with canonical maps of degree $10$, $11$ and $14$},
 https://arxiv.org/abs/2207.02969,
 (2022)
 %
\bibitem[FG21]{FG} D. Frapporti, C. Gleissner, 
\textit{Rigid manifolds of general type with non-contractible universal cover},
arXiv:2104.06775v2, (2021). 
%
\bibitem[G15]{G15}
   C. Gleissner,
   \textit{The classification of regular surfaces isogenous to a product of  curves with $\chi(\mathcal O_S)=2$},
Beauville surfaces and groups, Springer Proc. Math. Stat., 
     \textbf{123},
  Springer,  79--95, (2015).
  %
\bibitem[GPR18]{GPR}
C. Gleissner, R. Pignatelli, C.Rito, \textit{New surfaces with canonical
map of high degree}, to appear on Commun. Anal. Geom., arXiv:1807.11854, (2018). 
%
\bibitem[LY21]{LY21}
C.-J. Lai, S.-K.Yeung,  \textit{Examples of surfaces with canonical map of maximal degree}, Taiwanese J. Math. \textbf{4}, 699-716, (2021).
%
\bibitem[KM71]{KM71}
J.Morrow, K.Kodaira, \textit{Complex manifolds}, Holt, Rinehart and Winston, Inc., New York-Montreal, Que.-London, (1971). 
%
\bibitem[MLP21]{MLP21}
M.M. Lopes, R. Pardini, \textit{On the degree of the canonical map of a surface of general type}, 
arXiv e-prints,  arXiv:2103.01912v1, (2021). 
%
 \bibitem[N19]{Bin19}
B. Nguyen,  \textit{A new example of an algebraic surface with canonical map of degree 16}. Arch. Math. (Basel), \textbf{113}, 385--390, (2019). 
%
\bibitem[N21]{Bin21}
B. Nguyen, \textit{Some examples of algebraic surfaces with canonical map of degree 20}, 
Comptes Rendus Math\'{e}matique. Acad\'{e}mie des Sciences. Paris, \textbf{359}, 1145--1153, (2021). 
%
\bibitem[N22]{Bin22}
B. Nguyen, \textit{Some algebraic surfaces with canonical map of degree 10,12 and 14}, 
arXiv e-prints, arXiv:2207.04275, (2022). 
%
\bibitem[Pa91]{Pa91}
R.  Pardini,  \textit{Abelian covers of algebraic varieties}, J. Reine Angew. Math. \textbf{417}, 191--213,  (1991).
%
\bibitem[Per78]{Per}
U. Persson, \textit{Double coverings and surfaces of general type}, 
 In Algebraic geometry  vol. \textbf{687} of Lecture Notes in Math. Springer, 168--195,  (1978). 
%
%
\bibitem[Ri15]{Ri15}
C. Rito, \textit{New canonical triple covers of surfaces},  Proc. Amer. Math. Soc. \textbf{143}, 4647--653,  (2015). 
%
\bibitem[Ri17a]{Ri17}
C. Rito, \textit{A surface with canonical map of degree 24},   Internat. J. Math., \textbf{28},   (2017). 
%
\bibitem[Ri17b]{Ri17Zwei}
C. Rito, \textit{A surface with q = 2 and canonical map of degree 16},    Michigan Math. J., \textbf{66}, 99--105,  (2017). 
%
\bibitem[Ri22]{Ri22}
C. Rito, \textit{Surfaces with canonical map of maximum degree},  Journal of Algebraic Geometry, \textbf{31}, 127--135, (2022). 
%
\bibitem[S13]{Shafa}
I. Shafarevich, \textit{Basic algebraic geometry 1, Varieties in projective space}, Springer, Heidelberg, (2013). 
%
%
\bibitem[Tan03]{Tan} 
S. L. Tan,  \textit{Cusps on some algebraic surfaces and plane curves}, Complex Analysis, Complex Geometry and Related Topics - Namba, \textbf{60}, 106--121, (2003). 
\end{thebibliography}
\end{document}